\documentclass[a4paper,10pt,reqno]{amsart}

\usepackage[english]{babel}
\usepackage[utf8]{inputenc}
\usepackage{amsmath,amssymb,amsfonts,amsthm,amscd}
\usepackage[mathscr]{eucal}
\usepackage{enumitem}
\usepackage{tikz-cd}

\newcommand{\R}{\mathbb{R}}
\newcommand{\C}{\mathbb{C}}

\newcommand{\df}{\mathrm{d}}
\newcommand{\Nil}{\mathrm{Nil}_3}
\newcommand{\Iso}{\mathop{\rm Iso}\nolimits}

\newtheorem{theorem}{Theorem}[section]
\newtheorem{proposition}[theorem]{Proposition}
\newtheorem{corollary}[theorem]{Corollary}
\newtheorem{lemma}[theorem]{Lemma}

\theoremstyle{definition}

\theoremstyle{remark}
\newtheorem{remark}[theorem]{Remark}

\numberwithin{equation}{section}

\title[Dual quadratic differentials and entire minimal graphs]{Dual quadratic differentials and entire minimal graphs in Heisenberg space}
\date{}

\author{Jos\'{e} M. Manzano}
\address{Department of Mathematics\\
King's College London \\ Strand WC2R 2LS -- London (United Kingdom)}
\email{manzanoprego@gmail.com}

\thanks{This research has been partially supported by Spanish MEC-Feder Research Project MTM2014-52368-P, and by the EPSRC Grant No. EP/M024512/1. The author is grateful to Hojoo Lee because, when he heard about this manuscript, he sent to the author an unpublished preprint from 2009, in which he also proves Lemma~\ref{lema:hessian}, and other results closely related to Proposition~\ref{prop:bernstein}.}

\subjclass[2010]{Primary 53A10; Secondary 53C30}

\keywords{Minimal surfaces, constant mean curvature, homogeneous 3-manifolds, Heisenberg group, Lorentz--Minkowski space, quadratic differentials, Gauss curvature}

\begin{document}

\begin{abstract}
We define holomorphic quadratic differentials for spacelike surfaces with constant mean curvature in the Lorentzian homogeneous spaces $\mathbb{L}(\kappa,\tau)$ with isometry group of dimension 4, which are dual to the Abresch--Rosenberg differentials in the Riemannian counterparts $\mathbb{E}(\kappa,\tau)$, and obtain some consequences. On the one hand, we give a very short proof of the Bernstein problem in Heisenberg space, and provide a geometric description of the family of entire graphs sharing the same differential in terms of a 2-parameter conformal deformation. On the other hand, we prove that entire minimal graphs in Heisenberg space have negative Gauss curvature.
\end{abstract}

\maketitle

\section{Introduction}

The theory of constant mean curvature surfaces in homogeneous 3-manifolds has been a rather active research field during the last decades, and specially since Abresch and Rosenberg~\cite{AR,AR2} defined holomorphic quadratic differentials on constant mean curvature $H$ surfaces ($H$-surfaces in the sequel) on simply connected homogeneous Riemannian 3-manifolds $\mathbb{E}(\kappa,\tau)$ with isometry group of dimension 4, extending the classical Hopf differentials in space forms. Their target was the Hopf problem, and the existence of the differential revealed that $H$-spheres in $\mathbb{E}(\kappa,\tau)$ exist if and only if $4H^2+\kappa\leq 0$, in which case they are rotationally invariant. The value of $H$, if any, such that $4H^2+\kappa=0$ is usually called the \emph{critical} mean curvature. The Hopf problem has been solved very recently in the rest of homogeneous 3-manifolds by Meeks, Mira, P\'{e}rez, and Ros~\cite{MMPR}.

Given $\kappa,\tau\in\R$, the space $\mathbb{E}(\kappa,\tau)$ is characterized by admitting a Killing submersion with unitary Killing vector field and constant bundle curvature $\tau$ over the complete simply connected surface $\mathbb{M}^2(\kappa)$ with constant curvature $\kappa$, see~\cite{Dan,Man}. They include the product spaces $\mathbb{H}^2(\kappa)\times\R$ and $\mathbb{S}^2(\kappa)\times\R$, the Heisenberg space $\Nil(\tau)$, and $\mathrm{SU}(2)$ and $\widetilde{\mathrm{Sl}}_2(\R)$ endowed with special left-invariant metrics. There are homogeneous Lorentzian counterparts $\mathbb{L}(\kappa,\tau)$ enjoying the same description, but such that the fibers of the submersion are timelike, see~\cite{Lee,LeeMan}. It is important to notice that some space forms live also among them, namely the Euclidean space $\R^3=\mathbb{E}(0,0)$, the round spheres $\mathbb{S}^3(\kappa)=\mathbb{E}(\kappa,\frac{1}{2}\sqrt{\kappa})$, the Lorentz-Minkowski space $\mathbb{L}^3=\mathbb{L}(0,0)$, and the anti-de Sitter space $\mathbb{H}_1^3(\kappa)=\mathbb{L}(\kappa,\frac{1}{2}\sqrt{-\kappa})$.

Other milestones in the theory of $H$-surfaces in $\mathbb{E}(\kappa,\tau)$-spaces were the discovery by Daniel~\cite{Dan} of an isometric Lawson type correspondence, and the solution to the Bernstein problem in Heisenberg space by Fern\'{a}ndez and Mira~\cite{FM}. They showed that, for each holomorphic quadratic differential $Q$ on a non-compact simply connected Riemann surface $\Sigma$ (with $Q\neq 0$ if $\Sigma$ is parabolic), there is a 2-parameter family of entire minimal graphs (i.e., global minimal sections of the Killing submersion) in $\Nil(\frac{1}{2})=\mathbb{E}(0,\frac{1}{2})$ with Abresch--Rosenberg differential $Q$. Their solution was based on Wan and Au's result that the Hopf differential establishes a bijection between entire spacelike $\frac{1}{2}$-graphs in the Lorentz-Minkowski space $\mathbb{L}^3$ and holomorphic quadratic differentials~\cite{Wan,WanAu}, up to ambient isometries or conformal reparametrizations. In the way from $\frac{1}{2}$-graphs in $\mathbb{L}^3$ to minimal graphs in $\Nil(\frac{1}{2})$, Fern\'{a}ndez and Mira used the Daniel correspondence and the hyperbolic Gauss map for $\frac{1}{2}$-surfaces in $\mathbb{H}^2\times\R$ (see~\cite{FM3}), among other related techniques (we refer to the lecture notes by Daniel, Hauswirth and Mira~\cite{DHM}, as well as to the \textsc{icm} proceedings~\cite{FM2} for further reading). Dorfmeister, Inoguchi and Kobayashi~\cite{DIK} have obtained another proof of Bernstein problem by means of loop-group techniques, which also uses Wan and Au's result.

More recently, Lee~\cite{Lee} found a duality between $H$-graphs in $\mathbb{E}(\kappa,\tau)$ and spacelike $\tau$-graphs in $\mathbb{L}(\kappa,H)$ over simply connected subdomains of $\mathbb{M}^2(\kappa)$, which generalizes the classical Calabi duality~\cite{Calabi} between minimal surfaces in $\R^3$ and maximal surfaces in $\mathbb{L}^3$, and easily extends to spacelike conformal immersions with constant mean curvature which are not necessarily graphs, see~\cite{Man17}. The duality gives a reason why surfaces with critical mean curvature enjoy special properties: they are precisely those whose dual surfaces lie in constant sectional curvature spaces, where the geometry is richer. We obtain new applications of the duality connecting the conformal theories of $H$-surfaces in the Riemannian and Lorentzian settings. 

On the one hand, we define a holomorphic quadratic differential for spacelike $\tau$-surfaces in $\mathbb{L}(\kappa,H)$, dual to the Abresch--Rosenberg differential (see Theorem~\ref{thm:differential}). In particular, the Abresch--Rosenberg differential of minimal graphs in $\Nil(\tau)$ and the Hopf differential of $\tau$-graphs in $\mathbb{L}^3$ are shown to be dual. We will follow the approach in~\cite{DHM} and elude the use of conformal parameters, obtaining explicit expressions of the differential only depending on the functions defining the graphs. Once this is achieved, we get a short solution to the Bernstein problem in $\Nil(\frac{1}{2})$ by plugging directly the duality in Wan and Au's results, and avoiding the use of harmonic maps. Observe that quadratic differentials (and hence the corresponding harmonic maps) of $\frac{1}{2}$-graphs in $\mathbb{H}^2\times\R$, minimal graphs in $\Nil(\frac{1}{2})$, spacelike $\frac{1}{2}$-graphs in $\mathbb{L}^3$ and maximal graphs in $\mathbb{H}^3_1(-1)$, are essentially the same object in the underlying Riemann surface, so our approach turns out to be \emph{a posteriori} a really short shortcut in Fern\'{a}ndez and Mira's solution.

On the other hand, our technique implies that two entire minimal graphs in $\Nil(\tau)$ have the same Abresch--Rosenberg differential if and only their dual $\tau$-graphs are congruent in $\mathbb{L}^3$. Since the isometry group of $\Nil(\tau)$ is 4-dimensional whereas that of $\mathbb{L}^3$ is 6-dimensional, there is room for a 2-parameter family of non-congruent entire minimal graphs with the same differential. Consider the 2-dimensional subgroup $G$ of the stabilizer of the origin of $\mathbb{L}^3$ leaving invariant the hyperbolic plane $\mathbb{H}^2=\{(x,y,z)\in\mathbb{L}^3:z=(1+x^2+y^2)^{1/2}\}$ and fixing a point of the ideal boundary $\partial_\infty\mathbb{H}^2$, which consists of hyperbolic and parabolic rotations of $\mathbb{H}^2$, see Remark~\ref{rmk:G}. Then, the action of $G$ on an entire minimal graph $\Sigma\subset\Nil(\tau)$ (via duality) gives all isometry classes of surfaces with the same Abresch--Rosenberg differential as $\Sigma$ (see Theorem~\ref{thm:orbit}). 

The 2-parameter family of graphs sharing the same differential degenerates (it becomes 1-dimensional or 0-dimensional) if and only if the dual surface is invariant under parabolic or hyperbolic screw motions. In order to illustrate this, we will check that entire minimal graphs with zero differential are the minimal umbrellas, and parabolic surfaces with differential $c\,\df\zeta^2$, where $\zeta$ is a conformal parameter over $\C$, are precisely the invariant examples found by Figueroa, Mercuri and Pedrosa~\cite{FMP}. Although this was already known, see~\cite{AR2,Dan,FM}, our approach also shows that invariant examples are related by hyperbolic deformations. Some other invariant examples will be discussed throughout the paper, though very few examples of explicit entire minimal graphs are known so far (e.g., the examples constructed by Cartier~\cite{Cartier} or Nelli, Sa Earp and Toubiana~\cite{NST} fail to be explicit).

The last part of the paper will be devoted to apply the duality to prove that entire graphs with critical mean curvature in $\mathbb{E}(\kappa,\tau)$, $\kappa-4\tau^2\neq 0$, have negative Gauss curvature. To this effect, we shall use a curvature estimate for entire $\tau$-graphs in $\mathbb{L}^3$ due to Cheng and Yau~\cite{CY2}, see also~\cite{CT}. The proof of the curvature estimate will also rely on a fancy relation between the Hessian determinants of dual graphs, which follows from comparing the moduli of the dual differentials (see Lemma~\ref{lema:hessian} and Remark~\ref{rmk:q}). It is worth mentioning that the Gauss equation for $H$-surfaces in $\mathbb{E}(\kappa,\tau)$ reads
\begin{equation}\label{eqn:gauss}
K=\det(A)+\tau^2+(\kappa-4\tau^2)\nu^2,
\end{equation}
where $K$, $A$ and $\nu$ denote the Gauss curvature, the shape operator, and the angle function of the surface, respectively~\cite{Dan,ER2}. If $\kappa+4H^2=0$, then Equation~\eqref{eqn:gauss} gives the estimate $K\leq\tau^2+H^2$, which is not sharp for entire graphs. Nonetheless, there do exist complete properly embedded minimal surfaces in $\Nil(\tau)$ whose Gauss curvature changes sign (e.g., some helicoids, see Remark~\ref{rmk:catenoids-helicoids}).

By a simple comparison with the Euclidean plane, we deduce that entire minimal graphs in $\Nil(\tau)$ have at least quadratic area growth, and it is quadratic if and only if the graph has finite total curvature by Hartman and Li's results~\cite{Hartman,Li}. This allows us to rephrase in terms of curvature the conjecture posed by the author, P\'{e}rez and Rodr\'{i}guez in~\cite{MPR} that there are no entire minimal graphs in $\Nil(\tau)$ with quadratic area growth. We also get the necessary condition that an entire graph with intrinsic quadratic area growth in $\Nil(\tau)$ converges to a vertical plane along any diverging sequence, so the Gauss curvature and the angle function of the graph vanish at infinity, see Proposition~\ref{prop:tendToZero}. No monotonicity formula for minimal surfaces in $\Nil(\tau)$ has been found so far, which is one of the key techniques leading to the same result in $\R^3$, and Theorem~\ref{thm:curvature} is the first lower bound for the area growth for entire minimal graphs in $\Nil(\tau)$. A sharp upper bound was found by the author and Nelli~\cite{ManNel}, showing that such area growth is at most cubic.

\section{The conformal duality}

Throughout the paper $\mathbb{E}(\kappa,\tau)$ (resp. $\mathbb{L}(\kappa,\tau)$), with $\kappa,\tau\in\R$, will denote the only simply connected Riemannian (resp. Lorentzian) 3-manifold admitting a Killing submersion $\pi$ (resp. $\widetilde\pi$) with constant bundle curvature $\tau$ over the complete simply connected surface $\mathbb{M}^2(\kappa)$ with constant curvature $\kappa$, see~\cite{Dan,Man}. The duality is a natural correspondence between surfaces that can be stated as follows: 

\begin{theorem}[Conformal duality~\cite{Lee,Man17}]\label{thm:duality}
Let $\Sigma$ be a simply connected Riemann surface, and let $\tau,\kappa,H\in\R$. There is a duality between
\begin{enumerate}[label=(\alph*)]
 \item nowhere vertical conformal $H$-immersions $X:\Sigma\to\mathbb{E}(\kappa,\tau)$;
 \item spacelike conformal $\tau$-immersions $\widetilde X:\Sigma\to\mathbb{L}(\kappa,H)$.
\end{enumerate}
The dual immersions $X$ and $\widetilde X$ are determined up to a vertical translation, and satisfy $\pi\circ X=\widetilde\pi\circ\widetilde X$.
\end{theorem}

We will introduce coordinates in order to build up an appropriate framework to develop the duality (a coordinate free approach was given in~\cite{Man17}). Let us define \[\lambda(x,y)=\left(1+\tfrac{\kappa}{4}(x^2+y^2)\right)^{-1}\]
over the simply connected domain $\Omega_\kappa=\left\{(x,y)\in\R^2:1+\tfrac{\kappa}{4}(x^2+y^2)>0\right\}$. Then $(\Omega_\kappa,\lambda^2(\df x^2+\df y^2))$ has constant curvature $\kappa$, and the spaces $\mathbb{E}(\kappa,\tau)$ and $\mathbb{L}(\kappa,H)$ can be locally modeled as the manifold $\Omega_\kappa\times\R$ endowed with the metrics
\begin{equation}\label{eqn:ds}
\begin{aligned}
 \df s^2_{\mathbb{E}(\kappa,\tau)}&=\lambda^2(\df x^2+\df y^2)+(\df z^2+\tau\lambda(y\df x-x\df y))^2,\\
 \df s^2_{\mathbb{L}(\kappa,H)}&=\lambda^2(\df x^2+\df y^2)-(\df z^2-H\lambda(y\df x-x\df y))^2,
\end{aligned}
\end{equation}
where $\lambda$ stands for the conformal factor in the base surface $\mathbb{M}^2(\kappa)$, given by
\[\lambda(x,y,z)=\left(1+\tfrac{\kappa}{4}(x^2+y^2)\right)^{-1}.\]
If $\kappa\leq 0$, then~\eqref{eqn:ds} provides a global model for $\mathbb{E}(\kappa,\tau)$ and $\mathbb{L}(\kappa,H)$, whereas a whole fiber is omitted if $\kappa>0$ (in this case $\mathbb{E}(\kappa,\tau)$ and $\mathbb{L}(\kappa,H)$ are topologically $\mathbb{S}^2\times\R$ or $\mathbb{S}^3$). Since these coordinates will be used for local computation, this will not become actually an issue.

We can define global orthonormal frames $\{E_1,E_2,E_3\}$ and $\{\widetilde E_1,\widetilde E_2,\widetilde E_3\}$ for $\df s^2_{\mathbb{E}(\kappa,\tau)}$ and $\df s^2_{\mathbb{L}(\kappa,H)}$, where $\widetilde E_3$ is timelike, as
\begin{equation}\label{eqn:frame}
\begin{aligned}
 E_1&=\tfrac{1}{\lambda}\partial_x-\tau y\partial_z,& E_2&=\tfrac{1}{\lambda}\partial_y+\tau x\partial_z,&E_3&=\partial_z,\\
 \widetilde E_1&=\tfrac{1}{\lambda}\partial_x+H y\partial_z,&\widetilde E_2&=\tfrac{1}{\lambda}\partial_y-H x\partial_z& \widetilde E_3&=\partial_z.
\end{aligned}\end{equation}
The Killing submersion in this model is nothing but $(x,y,z)\mapsto(x,y)$, and the vertical Killing vector field is $\partial_z$. The Levi-Civita connection $\overline\nabla$ of $\mathbb{E}(\kappa,\tau)$ can be evaluated at the frame $\{E_1,E_2,E_3\}$ by means of Koszul formula, giving rise to 
\begin{equation}\label{eqn:levi-civita:R}
\begin{aligned}
\overline\nabla_{E_1}E_1&=\tfrac{\kappa}{2}yE_2,&
\overline\nabla_{E_1}E_2&=-\tfrac{\kappa}{2}yE_1+\tau E_3,&
\overline\nabla_{E_1}E_3&=-\tau E_2,\\
\overline\nabla_{E_2}E_1&=-\tfrac{\kappa}{2}xE_2-\tau E_3,&
\overline\nabla_{E_2}E_2&=\tfrac{\kappa}{2}xE_1,&
\overline\nabla_{E_2}E_3&=\tau E_1,\\
\overline\nabla_{E_3}E_1&=-\tau E_2,&
\overline\nabla_{E_3}E_2&=\tau E_1,&
\overline\nabla_{E_3}E_3&=0,
\end{aligned}
\end{equation}
The Levi-Civita connection $\widetilde\nabla$ of $\mathbb{L}(\kappa,H)$, admits a representation very similar to~\eqref{eqn:levi-civita:R}, since it also satisfies the same Koszul formula (see~\cite{Oneill}):
\begin{equation}\label{eqn:levi-civita:L}
\begin{aligned}
\widetilde\nabla_{\widetilde E_1}\widetilde E_1&=\tfrac{\kappa}{2}y\widetilde E_2,&
\widetilde\nabla_{\widetilde E_1}\widetilde E_2&=-\tfrac{\kappa}{2}y\widetilde E_1-H \widetilde E_3,&
\widetilde\nabla_{\widetilde E_1}\widetilde E_3&=-H \widetilde E_2,\\
\widetilde\nabla_{\widetilde E_2}\widetilde E_1&=-\tfrac{\kappa}{2}x\widetilde E_2+H \widetilde E_3,&
\widetilde\nabla_{\widetilde E_2}\widetilde E_2&=\tfrac{\kappa}{2}x\widetilde E_1,&
\widetilde\nabla_{\widetilde E_2}\widetilde E_3&=H \widetilde E_1,\\
\widetilde\nabla_{\widetilde E_3}\widetilde E_1&=-H \widetilde E_2,&
\widetilde\nabla_{\widetilde E_3}\widetilde E_2&=H \widetilde E_1,&
\widetilde\nabla_{\widetilde E_3}\widetilde E_3&=0.
\end{aligned}
\end{equation}

\subsection{The duality in coordinates}\label{sec:coordinates}
For any $u\in C^\infty(\Omega)$ over some open domain $\Omega\subset\Omega_\kappa$, we can define the vertical graph of $u$ over the section $(x,y)\mapsto(x,y,0)$, as the surface $\Sigma_u\subset\Omega_\kappa\times\R$ given by
\[\Sigma_u=\{(x,y,u(x,y)):(x,y)\in\Omega\}.\]
Observe that the duality given by Theorem~\ref{thm:duality} restricts to a duality between graphs over simply connected domains. If $\Sigma_u\subset\mathbb{E}(\kappa,\tau)$ and $\Sigma_v\subset\mathbb{L}(\kappa,H)$ are dual graphs over $\Omega\subseteq\mathbb{M}^2(\kappa)$ for some $u,v\in C^\infty(\Omega)$, let us define the quantities 
\begin{equation}
\begin{aligned} 
\alpha&=\tfrac{1}{\lambda}u_x+\tau y,&\beta&=\tfrac{1}{\lambda}u_y-\tau x,&\omega&=(1+\alpha^2+\beta^2)^{1/2},\\
\widetilde\alpha&=\tfrac{1}{\lambda}v_x-H y,&\widetilde\beta&=\tfrac{1}{\lambda}v_y+H x,&\widetilde\omega&=(1-\widetilde\alpha^2-\widetilde\beta^2)^{1/2}.
\end{aligned}
\end{equation} 
Then $\Sigma_u$ and $\Sigma_v$ are dual graphs if and only if they satisfy the so-called \emph{twin relations}~\cite{Lee}, namely the equivalent identities
\begin{equation}\label{eqn:twin}
 \begin{aligned}
  (\widetilde\alpha,\widetilde\beta)&=\Big(\frac{\beta}{\omega},\frac{-\alpha}{\omega}\Big)\quad\Longleftrightarrow\quad
  (\alpha,\beta)&=\Big(\frac{-\widetilde\beta}{\widetilde\omega},\frac{\widetilde\alpha}{\widetilde\omega}\Big),
 \end{aligned}
\end{equation}
in which case we also have $\widetilde\omega=\frac{1}{\omega}$.

Consider the global frames $\{e_1,e_2\}$ in $\Sigma_u$ and $\{\widetilde e_1,\widetilde e_2\}$ in $\Sigma_v$ defined as
\begin{equation}\label{eqn:e1e2}
\begin{aligned}
 e_1&=E_1+\alpha E_3,& \widetilde e_1&=\widetilde E_1+\widetilde\alpha\widetilde E_3,\\
 e_2&=E_2+\beta E_3,& \widetilde e_2&=\widetilde E_2+\widetilde\beta\widetilde E_3.
\end{aligned}\end{equation}
Therefore, the first fundamental forms of $\Sigma_u$ and $\Sigma_v$ in these frames take the form
 \begin{equation}\label{eqn:1ff}
 \begin{aligned}
  \mathrm{I}&\equiv\left(\begin{matrix}1+\alpha^2&\alpha\beta\\\alpha\beta&1+\beta^2\end{matrix}\right),&
  \widetilde{\mathrm{I}}&\equiv\left(\begin{matrix}1-\widetilde\alpha^2&-\widetilde\alpha\widetilde\beta\\-\widetilde\alpha\widetilde\beta&1-\widetilde\beta^2\end{matrix}\right),
  \end{aligned}\end{equation}
so the twin relations~\eqref{eqn:twin} imply that the two matrices in~\eqref{eqn:1ff} are proportional with $\widetilde{\mathrm{I}}=\omega^{-2}\mathrm{I}=\widetilde\omega^2\mathrm{I}$. This means that the global diffeomorphism
\begin{equation}\label{eqn:phi}
\Phi:\Sigma_u\to\Sigma_v,\qquad \Phi(x,y,u(x,y))=(x,y,v(x,y)).
\end{equation}
is conformal. In terms of conformal immersions $X:\Sigma\to\Sigma_u$ and $\widetilde X:\Sigma\to\Sigma_v$, the condition $\pi\circ X=\widetilde\pi\circ\widetilde X$ yields $\Phi=\widetilde X\circ X^{-1}$, so the metrics induced by dual immersions are related by the following identity:
\begin{equation}\label{eqn:conformal-factor}
X^*\df s^2_{\mathbb{E}(\kappa,\tau)}=\omega^2\widetilde X^*\df s^2_{\mathbb{L}(\kappa,H)}.
\end{equation}

From~\eqref{eqn:e1e2}, we deduce that the vector fields
 \begin{equation}\label{eqn:N}
 \begin{aligned}
   N&=-\frac{\alpha}{\omega}E_1-\frac\beta\omega E_2+\frac1\omega E_3,&\widetilde N&=\frac{\widetilde\alpha}{\widetilde\omega}\widetilde E_1+\frac{\widetilde\beta}{\widetilde\omega}\widetilde E_2+\frac{1}{\widetilde\omega} \widetilde E_3.
  \end{aligned}\end{equation}
are the upward-pointing unit normal vector fields to $\Sigma_u$ and $\Sigma_v$, respectively. Hence~\eqref{eqn:twin} implies that the so-called angle functions $\nu=\langle N,E_3\rangle=\omega^{-1}$ and $\widetilde\nu=\langle \widetilde N,\widetilde E_3\rangle=\widetilde\omega^{-1}$ are reciprocal. Moreover, $N$ (resp. $\widetilde N$) allows us to define a $\frac{\pi}{2}$-rotation in the tangent bundle as $JX=N\times X$ (resp. $JX=\widetilde N\times X$). Here the usual orientation in the underlying manifold $\R^3$ will be assumed, which is crucial in the definition of the wedge product $\times$, namely $\langle N\times X,Y\rangle=\det(N,X,Y)$ for all vector fields $X,Y$ in both Riemannian and Lorentzian cases, see~\cite{Lopez}.

In particular, from~\eqref{eqn:e1e2} and~\eqref{eqn:N} we deduce the following identities, which will be helpful in Section~\ref{sec:differential}:
 \begin{equation}\label{eqn:Je1Je2}
 \begin{aligned} 
   Je_1&=-\frac{\alpha\beta}{\omega}e_1+\frac{1+\alpha^2}{\omega}e_2,&J\widetilde e_1&=\frac{\widetilde\alpha\widetilde\beta}{\widetilde\omega}\widetilde e_1+\frac{1-\widetilde\alpha^2}{\widetilde\omega}\widetilde e_2,\\
   Je_2&=-\frac{1+\beta^2}{\omega}e_1+\frac{\alpha\beta}{\omega}e_2,&J\widetilde e_2&=-\frac{1-\widetilde\beta^2}{\widetilde\omega}\widetilde e_1-\frac{\widetilde\alpha\widetilde\beta}{\widetilde\omega}\widetilde e_2.
  \end{aligned}\end{equation}

The second fundamental forms of $\Sigma_u$ and $\Sigma_v$ are defined as $\sigma(X,Y)=\langle\overline\nabla_XY,N\rangle$ and $\widetilde\sigma(X,Y)=\langle\overline\nabla_XY,\widetilde N\rangle$, respectively. Using~\eqref{eqn:e1e2},~\eqref{eqn:1ff},~\eqref{eqn:N}, and the Levi-Civita connection~\eqref{eqn:levi-civita:R}, we reach the following identities for $\sigma$:
\begin{equation}\label{eqn:2ff:R}
\begin{aligned}
\sigma(e_1,e_1)&=\frac{1}{\omega}\left(\frac{u_{xx}}{\lambda^2}+2\tau\alpha\beta+\frac{\kappa}{2}(x\alpha-y\beta)-\frac{\kappa\tau}{2}xy\right),\\
\sigma(e_1,e_2)&=\frac{1}{\omega}\left(\frac{u_{xy}}{\lambda^2}+\tau(\beta^2-\alpha^2)+\frac{\kappa}{2}(x\beta+y\alpha)+\frac{\kappa\tau}{4}(x^2-y^2)\right),\\
\sigma(e_2,e_2)&=\frac{1}{\omega}\left(\frac{u_{yy}}{\lambda^2}-2\tau\alpha\beta-\frac{\kappa}{2}(x\alpha-y\beta)+\frac{\kappa\tau}{2}xy\right).
\end{aligned}
\end{equation}
Analogously, we get the corresponding expressions for $\widetilde\sigma$ by means of~\eqref{eqn:levi-civita:L}:
\begin{equation}\label{eqn:2ff:L}
\begin{aligned}
\widetilde \sigma(\widetilde e_1,\widetilde e_1)&=\frac{-1}{\widetilde \omega}\left(\frac{v_{xx}}{\lambda^2}+2H\widetilde \alpha\widetilde \beta+\frac{\kappa}{2}(x\widetilde \alpha-y\widetilde \beta)+\frac{\kappa H}{2}xy\right),\\
\widetilde \sigma(\widetilde e_1,\widetilde e_2)&=\frac{-1}{\widetilde \omega}\left(\frac{v_{xy}}{\lambda^2}+H(\widetilde \beta^2-\widetilde \alpha^2)+\frac{\kappa}{2}(x\widetilde \beta+y\widetilde \alpha)-\frac{\kappa H}{4}(x^2-y^2)\right),\\
\widetilde \sigma(\widetilde e_2,\widetilde e_2)&=\frac{-1}{\widetilde \omega}\left(\frac{v_{yy}}{\lambda^2}-2H\widetilde \alpha\widetilde \beta-\frac{\kappa}{2}(x\widetilde \alpha-y\widetilde \beta)-\frac{\kappa H}{2}xy\right).
\end{aligned}
\end{equation}

\subsection{Holomorphic quadratic differentials}\label{sec:differential}
Given arbitrary $\kappa,\tau,H\in\R$, let us consider constants
$a,\widetilde a,b,\widetilde b\in\C$ satisfying the following linear identities
\begin{equation}\label{eqn:ab}
\begin{aligned}
(\kappa-4\tau^2)a+2(H+i\tau)b&=0,\\
(\kappa+4H^2)\widetilde a+2i(H+i\tau)\widetilde b&=0,
\end{aligned}
\end{equation}
We define a quadratic differential $Q_{a,b}$ for immersed $H$-surfaces in $\mathbb{E}(\kappa,\tau)$, and a quadratic differential $\widetilde Q_{\widetilde a,\widetilde b}$ for immersed spacelike $\tau$-surfaces in $\mathbb{L}(\kappa,H)$ as
\begin{equation}\label{eqn:Q}
\begin{aligned}
Q_{a,b}(X,Y)=a\,\sigma(X-iJX,Y-iJY)+b\langle X-iJX,E_3\rangle\langle Y-iJY,E_3\rangle,\\
\widetilde Q_{\widetilde a,\widetilde b}(X,Y)=\widetilde a\,\widetilde\sigma(X-iJX,Y-iJY)+\widetilde b\langle X-iJX,\widetilde E_3\rangle\langle Y-iJY,\widetilde E_3\rangle,
\end{aligned}
\end{equation}
i.e., $Q_{a,b}$ and $\widetilde Q_{\widetilde a,\widetilde b}$ are the $(2,0)$-components of linear combinations of the complexified second fundamental forms (denoted by $\sigma$ and $\widetilde\sigma$, respectively), and the height differentials. Recall that $J$ is the $\frac\pi2$-rotation defined in the previous section.

The differential $Q_{a,b}$ is nothing but the Abresch--Rosenberg differential of an immersed $H$-surface in $\mathbb{E}(\kappa,\tau)$, see~\cite[Theorem~2.2.1]{DHM}, which is holomorphic. The definition of $\widetilde Q_{\widetilde a,\widetilde b}$ is motivated by the duality, as shown by Theorem~\ref{thm:duality} below.

\begin{remark}
Changing the sign of $N$ or $\widetilde N$ in~\eqref{eqn:Q} implies a change of the signs of $J$ and of the second fundamental form, which gives other (non-holomorphic) quadratic differentials. In other words, some compatibility in the orientation is needed, in the sense that $JX=N\times X$ must hold true when we choose the orientation in the ambient space for which the bundle curvature is $\tau$ (note that the cross product $\times$ and the sign of $\tau$ depend upon the choice of orientation,  see also~\cite{Man}).
\end{remark}

\begin{theorem}\label{thm:differential}
Let $\kappa,\tau,H\in\R$, and let $X:\Sigma\to\mathbb{E}(\kappa,\tau)$ and $\widetilde X:\Sigma\to\mathbb{L}(\kappa,H)$ be dual conformal immersions. Given $a,\widetilde a,b,\widetilde b\in\C$ satisfying~\eqref{eqn:ab}, it follows that 
\[X^*Q_{a,b}=\widetilde X^*\widetilde Q_{\widetilde a,\widetilde b},\]
whenever $\widetilde a=-ia$ and $(\kappa-4\tau^2)\widetilde b+(\kappa+4H^2)b=0$.
\end{theorem}

\begin{remark}[special cases]\label{rmk:degenerate}
Note that $\kappa-4\tau^2$ and $\kappa+4H^2$ cannot be zero simultaneously unless $\kappa=\tau=H=0$, so the relation between $(a,b)$ and $(\widetilde a,\widetilde b)$ is well defined. In the \emph{degenerate} case $\kappa=\tau=H=0$, the conditions in~\eqref{eqn:ab} become trivial, but we get the original Calabi duality between minimal graphs in $\R^3$ and maximal graphs in $\mathbb{L}^3$. It is well known~\cite{AL} that the classical Hopf differentials in $\R^3$ and $\mathbb{L}^3$ agree via the duality. This will be discussed in Remark~\ref{rmk:differential:nil} below.

If $\kappa+4H^2=0$, then $X$ has critical constant mean curvature, and $\mathbb{L}(\kappa,\tau)$ becomes a Lorentzian space form, and the Abresch--Rosenberg differential $Q_{a,b}$ corresponds to the classical Hopf differential $Q_{-ia,0}$ in $\mathbb{L}^3$ or $\mathbb{H}^3_1(\kappa)$. Likewise, if $\kappa-4\tau^2=0$, then $\mathbb{E}(\kappa,\tau)$ becomes a Riemannian space form and $b=0$, i.e., $Q_{a,b}$ is (a multiple of) the Hopf differential in $\R^3$ or $\mathbb{S}^3(\kappa)$.
\end{remark}

\begin{proof}[Proof of Theorem~\ref{thm:differential}]
Since the result is local, we will consider dual graphs $\Sigma_u\subset\mathbb{E}(\kappa,\tau)$ and $\Sigma_v\subset\mathbb{L}(\kappa,H)$ over some common domain of $\mathbb{M}^2(\kappa)$. Following the notation in Section~\ref{sec:coordinates}, it suffices to check that $Q_{a,b}(e_i,e_j)=\widetilde Q_{\widetilde a,\widetilde b}(\widetilde e_i,\widetilde e_j)$ for all $i,j\in\{1,2\}$. Since~\eqref{eqn:ab} are linear relations, we will choose $a=2(H+i\tau)$, $b=-(\kappa-4\tau^2)$, $\widetilde a=2(\tau-iH)$, and $\widetilde b=\kappa+4H^2$, see Remark~\ref{rmk:degenerate}, and the subindices $(a,b)$ and $(\widetilde a,\widetilde b)$ will be omitted in the sequel.

The computations are somewhat cumbersome, so we will just sketch the proof of $Q(e_1,e_1)=\widetilde Q(\widetilde e_1,\widetilde e_1)$. Since $\sigma(JX,JY)=2H\langle X,Y\rangle-\sigma(X,Y)$ for all vector fields $X$ and $Y$ on $\Sigma_u$, we obtain from Equation~\eqref{eqn:Q} that
\begin{equation}\label{thm:differential:eqn1}
\begin{aligned}
Q(e_1,e_1)&=4(H+i\tau)\big(\sigma(e_1,e_1)-i\sigma(e_1,Je_1)-H\langle e_1,e_1\rangle\big)\\
&\qquad-(\kappa-4\tau^2)\big(\langle e_1,E_3\rangle-i\langle Je_1,E_3\rangle\big)^2.
\end{aligned}
\end{equation}
We can make use of~\eqref{eqn:e1e2},~\eqref{eqn:1ff},~\eqref{eqn:Je1Je2}, and~\eqref{eqn:2ff:R} to evaluate the right-hand-side of~\eqref{thm:differential:eqn1}, getting to an expression of $Q(e_1,e_1)$ in terms of $u_{xx}$, $u_{xy}$, and the first-order symbols $\alpha$ and $\beta$. On the one hand, $\alpha$ and $\beta$ can be written directly in terms of $\widetilde\alpha$ and $\widetilde\beta$ by means of the twin relations~\eqref{eqn:twin}, which give the following explicit first-order \textsc{pde} system relating $u$ and $v$ (see also~\cite{Lee,Man17}):
\begin{align}
 \alpha=\frac{\widetilde\beta}{\widetilde\omega}\quad\Leftrightarrow\quad u_x&=\frac{v_y+Hx\lambda}{\sqrt{1-(\frac{v_x}{\lambda}-H y)^2-(\frac{v_y}{\lambda}+H x)^2}}-\tau y\lambda,\label{thm:differential:eqn2}\\
 \beta=\frac{-\widetilde\alpha}{\widetilde\omega}\quad\Leftrightarrow\quad u_y&=\frac{-v_x+Hy\lambda}{\sqrt{1-(\frac{v_x}{\lambda}-H y)^2-(\frac{v_y}{\lambda}+H x)^2}}+\tau x\lambda.\label{thm:differential:eqn3}
 \end{align}
Therefore $u_{xx}$ (resp. $u_{xy}$) can be worked out by differentiating~\eqref{thm:differential:eqn2} (resp.~\eqref{thm:differential:eqn3}) with respect to $x$, which gives an expression in terms of $v_{xx}$, $v_{xy}$, $\widetilde\alpha$ and $\widetilde\beta$. Plugging such expressions for $u_{xx}$ and $u_{xy}$ in the result of evaluating~\eqref{thm:differential:eqn1}, we get
\begin{equation}\label{thm:differential:eqn4}
Q(e_1,e_1)=\frac{4(H+i\tau)\lambda^2(\widetilde\alpha\widetilde\beta+i\widetilde\omega)}{\widetilde\omega^2}v_{xx}+\frac{4(H+i\tau)\lambda^2(1-\widetilde\alpha^2)}{\widetilde\omega^2}v_{xy}+L,
\end{equation}
where $L$ represents the lower order terms.

The proof will finish if we check that~\eqref{thm:differential:eqn4} coincides with $\widetilde Q(\widetilde e_1,\widetilde e_1)$. Using~\eqref{eqn:Q} and the fact that $\widetilde\sigma(JX,JY)=-2\tau\langle X,Y\rangle-\widetilde\sigma(X,Y)$ for all vector fields $X$ and $Y$ on $\Sigma_v$ (note the change of sign with respect to the Riemannian case due to the definition of mean curvature in the Lorentzian setting, see, e.g.,~\cite{Lopez}), we reach
\begin{equation}\label{thm:differential:eqn5}
\begin{aligned}
\widetilde Q(\widetilde e_1,\widetilde e_1)&=4(\tau-iH)\big(\widetilde \sigma(\widetilde e_1,\widetilde e_1)-i\widetilde \sigma(\widetilde e_1,J\widetilde e_1)+\tau\langle \widetilde e_1,\widetilde e_1\rangle\big)\\
&\qquad+(\kappa+4H^2)\big(\langle \widetilde e_1,\widetilde E_3\rangle-i\langle J\widetilde e_1,\widetilde E_3\rangle\big)^2.
\end{aligned}
\end{equation}
Transforming the right-hand-side of~\eqref{thm:differential:eqn5} by means of~\eqref{eqn:e1e2},~\eqref{eqn:1ff},~\eqref{eqn:Je1Je2}, and~\eqref{eqn:2ff:R}, it is long but straightforward to verify that it is equal to~\eqref{thm:differential:eqn4}.
\end{proof}

\begin{corollary}
Given $\widetilde a,\widetilde b\in\C$ satisfying~\eqref{eqn:ab}, the differential $\widetilde Q_{\widetilde a,\widetilde b}$ is holomorphic on spacelike $\tau$-surfaces immersed in $\mathbb{L}(\kappa,H)$.
\end{corollary}

\begin{remark}\label{rmk:differential:nil}
In the case of a minimal graph $\Sigma_u\subset\mathbb{E}(0,\tau)$ (i.e., in the Heisenberg space $\Nil(\tau)$ or in $\R^3$), let $Q=Q_{1,-2i\tau}$. Evaluating $Q$ at $\{e_1=\partial_x+u_x\partial_z,e_2=\partial_y+u_x\partial_z\}$ as in the proof of Theorem~\ref{thm:differential}, we get the simple expressions
\begin{equation}\label{eqn:Qei:R}
\begin{aligned}
Q(e_1,e_1)&=\frac{2u_{xx}}{\omega}+\frac{2i}{\omega^2}\left(\alpha\beta u_{xx}-(1+\alpha^2)u_{xy}\right),\\
Q(e_1,e_2)&=\frac{2u_{xy}}{\omega}+\frac{2i}{\omega^2}\left((1+\beta^2)u_{xx}-\alpha\beta u_{xy}\right),\\
Q(e_2,e_2)&=\frac{2u_{yy}}{\omega}+\frac{2i}{\omega^2}\left((1+\beta^2) u_{xy}-\alpha\beta u_{yy}\right).
\end{aligned}
\end{equation}
The Hopf differential $\widetilde Q=\widetilde Q_{-i,0}$ of the dual $\tau$-graph $\Sigma_v\subset\mathbb{L}^3$ at the dual frame $\{\widetilde e_1=\partial_x+v_x\partial_z,\widetilde e_2=\partial_y+v_y\partial_z\}$ satisfies the formulas
\begin{equation}\label{eqn:Qei:L}
\begin{aligned}
\widetilde Q(\widetilde e_1,\widetilde e_1)&=\frac{2}{\widetilde\omega^2}\left(\widetilde\alpha\widetilde\beta v_{xx}+(1-\widetilde\alpha^2)v_{xy}\right)+\frac{2i}{\widetilde\omega}v_{xx}-2i\tau(1-\widetilde\alpha^2),\\
\widetilde Q(\widetilde e_1,\widetilde e_2)&=\frac{-2}{\widetilde\omega^2}\left((1-\widetilde\beta^2)v_{xx}+\widetilde\alpha\widetilde\beta v_{xy}\right)+\tau\widetilde\omega+\frac{2i}{\widetilde\omega}v_{xy}+2i\tau\widetilde\alpha\widetilde\beta,\\
\widetilde Q(\widetilde e_2,\widetilde e_2)&=\frac{-2}{\widetilde\omega^2}\left((1-\widetilde\beta^2)v_{xx}+\widetilde\alpha\widetilde\beta v_{xy}\right)+\frac{2i}{\widetilde\omega}v_{yy}-2i\tau(1-\widetilde\beta^2).
\end{aligned}
\end{equation}
It follows that $Q_{1,-2\tau i}=\Phi^*\widetilde Q_{-i,0}$, where $\Phi:\Sigma_u\to\Sigma_v$ is the conformal diffeomorphism defined in~\eqref{eqn:phi}. If $\tau=0$, then the classical Hopf differentials $Q_{1,0}$ and $\widetilde Q_{1,0}$ of dual graphs in $\R^3$ and $\mathbb{L}^3$ are related by $Q_{1,0}=\Phi^*\widetilde Q_{-i,0}=-i\Phi^*\widetilde Q_{1,0}$.

The minimal surface equation for $\Sigma_u\subset\Nil(\tau)$ and the mean curvature $\tau$ equation for $\Sigma_v\subset\mathbb{L}^3$, which are the integrability conditions for the twin relations and will be used below, can be written as
\begin{equation}\label{eqn:mean-curvature}
\begin{aligned}
 (1+\beta^2)u_{xx}-2\alpha\beta u_{xy}+(1+\alpha^2)u_{yy}&=0,\\
 (1-\widetilde\beta^2)v_{xx}+2\widetilde\alpha\widetilde\beta v_{xy}+(1-\widetilde\alpha^2)v_{yy}&=2\tau\widetilde\omega^3.
\end{aligned}
\end{equation}
\end{remark}

\subsection{Behavior under isometric deformations} 

Consider the standard Killing submersion structures $\pi:\mathbb{E}(\kappa,\tau)\to\mathbb{M}^2(\kappa)$ and $\widetilde\pi:\mathbb{L}(\kappa,H)\to\mathbb{M}^2(\kappa)$, and denote by $\Iso_\xi^+(\mathbb{E}(\kappa,\tau))$ and $\Iso_\xi^+(\mathbb{L}(\kappa,H))$ the subgroups of $\Iso(\mathbb{E}(\kappa,\tau))$ and $\Iso(\mathbb{L}(\kappa,\tau))$ of direct Killing isometries leaving the unit Killing vector field $\xi$ invariant, see~\cite{Man}.

For each $T\in\Iso_\xi^+(\mathbb{E}(\kappa,\tau))$ and $\widetilde T\in\Iso_\xi^+(\mathbb{L}(\kappa,H))$, there are unique direct isometries $\rho(T),\widetilde\rho(\widetilde T)\in\Iso^+(\mathbb{M}^2(\kappa))$ such that $\rho(T)\circ\pi=\pi\circ T$ and $\widetilde\rho(\widetilde T)\circ\widetilde\pi=\widetilde\pi\circ \widetilde T$. Therefore, $\rho$ and $\widetilde\rho$ define group epimorphisms with kernel the subgroups of \emph{vertical translations} (i.e., the isometries spanned by $\xi$), so there is a group isomorphism 
\begin{equation}\label{eqn:R}
R:\frac{\Iso_\xi^+(\mathbb{E}(\kappa,\tau))}{\ker(\rho)}\longrightarrow\frac{\Iso_\xi^+(\mathbb{L}(\kappa,\tau))}{\ker(\widetilde\rho)}.
\end{equation}
Both the duality and the isomorphism $R$ ignore vertical translations, so the following result makes sense. The proof can be found at~\cite{Man17}.

\begin{proposition}\label{prop:transform}
Let $X:\Sigma\to\mathbb{E}(\kappa,\tau)$ and $\widetilde X:\Sigma\to\mathbb{L}(\kappa,H)$ be dual conformal immersions, and let $T\in\mathrm{Iso}_\xi^+(\mathbb{E}(\kappa,\tau))$. Then $T\circ X$ and $R(T)\circ\widetilde X$ are also dual.
\end{proposition}

\begin{remark}\label{rmk:transform}
Proposition~\ref{prop:transform} generically describes the whole picture concerning Killing isometries~\cite{Man}, since the isometries $T$ such that $T_*\xi=-\xi$, which project to isometries of $\Iso^-(\mathbb{M}^2(\kappa))$, change the sign of the mean curvature (we are considering the upward-pointing normal), and the duality does not apply. 

If $H=0$ or $\tau=0$, then this is not actually an issue, and we can complete the picture by considering the following exceptional cases:
	\begin{enumerate}
		\item If $H=0$, then symmetries about a horizontal geodesics in $\mathbb{E}(\kappa,\tau)$ correspond to mirror symmetries about vertical planes in $\mathbb{L}(\kappa,0)=\mathbb{M}^2(\kappa)\times\R_1$. 
		\item Likewise, if $\tau=0$, mirror symmetries about vertical planes in $\mathbb{E}(\kappa,0)=\mathbb{M}^2(\kappa)\times\R$ correspond to symmetries about horizontal geodesics in $\mathbb{L}(\kappa,H)$. 
		\item If $H=\tau=0$, then we get dual surfaces with zero mean curvature in product spaces, and mirror symmetries about horizontal planes in $\mathbb{E}(\kappa,\tau)$ correspond to mirror symmetries about horizontal planes in $\mathbb{L}(\kappa,\tau)$
	\end{enumerate}
In particular, for minimal surfaces in the Heisenberg group ($\kappa=H=0$), Proposition~\ref{prop:transform} still holds true after extending the isomorphism $R$ to an isomorphism 
\[\overline R:\frac{\Iso(\Nil(\tau))}{\ker(\rho)}\longrightarrow\frac{\Iso_\xi(\mathbb{L}^3)}{\ker(\widetilde\rho)}.\]
\end{remark}
The following consequence of Proposition~\ref{prop:transform} will be helpful when analyzing some examples invariant under a 1-parameter group of isometries in the next section:

\begin{corollary}\label{coro:transform}
$X$ is invariant under a subgroup of $\Iso_\xi^+(\mathbb{E}(\kappa,\tau))$ if and only if $\widetilde X$ is invariant under a subgroup of $\Iso_\xi^+(\mathbb{L}(\kappa,H))$, in which case both subgroups induce the same subgroup of $\Iso(\mathbb{M}^2(\kappa))$ via $\rho$ and $\widetilde\rho$.
\end{corollary}

\section{The Bernstein problem}

As pointed out in the introduction, we will give here a direct proof of Fern\'{a}ndez and Mira's solution to the Bernstein problem in $\Nil(\tau)$, plus a geometric description of the families of surfaces sharing the same Abresch--Rosenberg differential. Essentially, we are interested in parameterizing the moduli space of entire minimal graphs in $\Nil(\tau)$ up to ambient isometries
\begin{equation}\label{eqn:moduli-space}
\frac{\mathcal{E}(\Nil(\tau),0)}{\Iso(\Nil(\tau))},
\end{equation}
where $\mathcal{E}(X,H)$ stands for the space of entire spacelike $H$-graphs in $X\in\{\Nil(\tau),\mathbb{L}^3\}$ up to vertical translations, which is a natural space where the duality applies. 

The 4-dimensional group $\Iso_\xi(\Nil(\tau))$ acts on $\mathcal{E}(\Nil(\tau),0)$ in the usual way, whereas on $\mathcal{E}(\mathbb{L}^3,\tau)$ we have the action of the 6-dimensional group $\Iso_\uparrow(\mathbb{L}^3)$ of isometries preserving the time orientation (which contains the 4-dimensional subgroup $\Iso_\xi(\mathbb{L}^3)$). Cheng and Yau~\cite{CY2} proved that a spacelike surface in $\mathbb{L}^3$ with constant mean curvature $\tau$ is an entire graph if and only if it is complete, and hence each element of $\Iso_\uparrow(\mathbb{L}^3)$ preserves the condition of being an entire graph.

On the one hand, Proposition~\ref{prop:transform} and Remark~\ref{rmk:transform} yield a bijective duality between~\eqref{eqn:moduli-space} and $\mathcal{E}(\mathbb{L}^3,\tau)/\Iso_\xi(\mathbb{L}^3)$. On the other hand, in the conformal level, the Abresch--Rosenberg differential defines a map $\mathcal{Q}:\mathcal{E}(\Nil(\tau),0)/\Iso(\Nil(\tau))\to\mathbf{QD}$, where $\mathbf{QD}$ stands for the set of holomorphic quadratic differentials on the disk $\mathbb{D}=\{z\in\C:|z|\leq 1\}$ (hyperbolic case) or the complex plane $\C$ (parabolic case), up to conformal changes of variables. We will exclude from $\mathbf{QD}$ the zero differential in the parabolic case due to the following non-existence result, see also~\cite{AR2,Dan2}:

\begin{lemma}\label{lema:parabolic-Q0}
There are no entire parabolic minimal graphs in $\Nil(\tau)$ with zero Abresch--Rosenberg differential.
\end{lemma}

\begin{proof}
By Theorem~\ref{thm:differential}, the dual of a graph with zero Abresch--Rosenberg differential has zero Hopf differential, and hence it is totally umbilical. Therefore it must be congruent to the hyperboloid $\Sigma_v$ with $v(x,y)=(\tau^{-2}+x^2+y^2)^{1/2}$. Since the hyperboloid is a complete surface with constant negative Gauss curvature, it is clearly hyperbolic, and so is the original surface.
\end{proof}

\begin{figure}
\[\begin{tikzcd}
 \mathcal{E}(\Nil(\tau),0)
 	\arrow[rr,rightarrow]&&
 \frac{\mathcal{E}(\Nil(\tau),0)}{\Iso(\Nil(\tau))}
	\arrow[rrd,rightarrow, "\Lambda"]
	\arrow[rr, "\mathcal{Q}"]&&
 \mathbf{QD}
 	\arrow[d,leftrightarrow ,"\text{Wan-Au}"]\\
 \mathcal{E}(\mathbb{L}^3,\tau)
 	\arrow[u,leftrightarrow, "\text{duality}"]
 	\arrow[rr,rightarrow]&&
 \frac{\mathcal{E}(\mathbb{L}^3,\tau)}{\Iso_\xi(\mathbb{L}^3)}
 	\arrow[u,leftrightarrow, "{\begin{matrix}\text{induced}\vspace{-5pt}\\ \text{duality}\end{matrix}}"]\arrow[rr,rightarrow]&&
 \frac{\mathcal{E}(\mathbb{L}^3,\tau)}{\Iso_\uparrow(\mathbb{L}^3)}
\end{tikzcd}\]
\caption{Relations between entire graphs and quadratic differentials.}\label{fig:diagram}
\end{figure}

Wan and Au~\cite{Wan,WanAu} proved that there Hopf differential defines a bijection from $\mathcal{E}(\mathbb{L}^3,\tau)/\Iso_\uparrow(\mathbb{L}^3)$ to $\mathbf{QD}$. Since the Hopf and the Abresch--Rosenberg differentials are dual by Theorem~\ref{thm:differential}, we infer that $\mathcal Q$ induces a surjective map 
\[\Lambda:\frac{\mathcal{E}(\Nil(\tau)),0)}{\Iso(\Nil(\tau))}\longrightarrow\frac{\mathcal{E}(\mathbb{L}^3,\tau)}{\Iso_{\uparrow}(\mathbb{L}^3)}\]
such that the diagram in Figure~\ref{fig:diagram} is commutative. This means that the families of isometry classes of surfaces with the same differential are precisely the level sets of $\Lambda$, so we get the desired solution to the Bernstein problem:

\begin{proposition}[Bernstein problem]\label{prop:bernstein}
For each $Q\in\mathbf{QD}$ there exists an entire minimal graph $\Sigma\subset\Nil(\tau)$ with Abresch--Rosenberg differential $Q$. Another entire minimal graph $\Sigma'\subset\Nil(\tau)$ has the same Abresch--Rosenberg differential as $\Sigma$ if and only if the dual graphs of $\Sigma$ and $\Sigma'$ are congruent in $\mathbb{L}^3$.
\end{proposition}

Motivated by the last sentence in the above statement, if $S\in\Iso_\uparrow(\mathbb{L}^3)$ and $\Sigma_u\in\mathcal{E}(\Nil(\tau),0)$ with dual graph $\Sigma_v\in\mathcal{E}(\mathbb{L}^3,\tau)$, then $[\Sigma_u],[\Sigma_u]_S\in\frac{\mathcal{E}(\Nil(\tau),0)}{\Iso(\Nil(\tau))}$ will denote the isometry class of $\Sigma_u$ and of a dual of $S(\Sigma_v)$, respectively. Therefore, Proposition~\ref{prop:bernstein} tells us that the set of isometry classes with the same Abresch--Rosenberg differential as $\Sigma_u$ is given by
\begin{equation}\label{eqn:level-sets}
\Lambda^{-1}(\Lambda([\Sigma_u]))=\{[\Sigma_u]_S:S\in\Iso_\uparrow(\mathbb{L}^3)\}.
\end{equation}
Nonetheless, this description is too redundant since $\Iso_\uparrow(\mathbb{L}^3)$ is 6-dimensional, and Proposition~\ref{prop:transform} implies that $[\Sigma_u]_S=[\Sigma_u]$ when $S$ is a Killing isometry. Our goal is to prove that~\eqref{eqn:level-sets} still holds when one substitutes $\Iso_\uparrow(\mathbb{L}^3)$ with the 2-dimensional subgroup $G$ spanned by the following two 1-parameter groups of isometries of $\mathbb{L}^3$:
\begin{itemize}
	\item hyperbolic rotations, given in terms of a parameter $\theta\in\R$ by
		\begin{equation}\label{eqn:hyperbolic-rotation}
    \phi(x,y,z)=(x,y\cosh(\theta)+z\sinh(\theta),y\sinh(\theta)+z\cosh(\theta));
    \end{equation}
	\item parabolic rotations, given in terms of a parameter $a\in\R$ by
		\begin{equation}\label{eqn:parabolic-rotation}
    \phi(x,y,z)=(x-ay+az,ax+(1-\tfrac{a^2}{2})y+\tfrac{a^2}{2}z,ax-\tfrac{a^2}{2}y+(1+\tfrac{a^2}{2})z).
    \end{equation}
\end{itemize}

\begin{remark}\label{rmk:G}
Since $G$ is contained in the stabilizer of the origin of $\mathbb{L}^3$, it can be identified with a subgroup of isometries $\mathbb{H}^2=\{(x,y,z)\in\mathbb{L}^3:z=(1+x^2+y^2)^{1/2}\}$, the hyperbolic plane. Via the isometry between the hyperboloid model $\mathbb{H}^2$ and the Poincar\'e's half-space model $\R^2_+=\{(x,y)\in\R^2:y>0\}$ given by
\begin{align*}
\psi&:\mathbb{H}^2\to\R^2_+,&\psi(x,y,z)&=\left(\frac{2 x
   (1+z)}{x^2+(1-y+z)^2},\frac{(1+z)^2-x^2-y^2}{x^2+(1-y+z)^2}\right),
\end{align*}
the above hyperbolic and parabolic rotations correspond to $(x,y)\mapsto e^\theta(x,y)$ and $(x,y)\mapsto(x+a,y)$, respectively. Hence $G$ is isomorphic to the 2-parameter group of direct isometries of $\mathbb{H}^2$ fixing a point $\infty$ of the ideal boundary $\partial_\infty\mathbb{H}^2$.\end{remark}

\begin{lemma}\label{lemma:orbit1}
For each $S\in\Iso_{\uparrow}(\mathbb{L}^3)$, there exist unique $S_1\in\Iso_\xi(\mathbb{L}^3)$ and $S_2\in G$ such that $S=S_1\circ S_2$.
\end{lemma}

\begin{proof}
Let $S_3\in\Iso_\xi(\mathbb{L}^3)$ be the translation mapping $S(0,0,0)$ to $(0,0,0)$, so $S_3\circ S$ lies in the stabilizer of $(0,0,0)$ and induces an isometry of $\mathbb{H}^2$. Hence there exists $S_4\in\Iso_\xi(\mathbb{L}^3)$ such that $S_4((S_3\circ S)(\infty))=\infty$. Then $S_1=S_3^{-1}\circ S_4^{-1}\in\Iso_\xi(\mathbb{L}^3)$ and $S_2=S_4\circ S_3\circ S\in G$ satisfy the conditions in the statement. 

As for uniqueness, let us assume that $S_1,S_1'\in\Iso_\xi(\mathbb{L}^3)$ and $S_2,S_2'\in G$ are such that $S'_1\circ S'_2=S_1\circ S_2$. Evaluating at $(0,0,0)$, we get that $S_1(0,0,0)=S_1'(0,0,0)$, so there is a translation $S_5$ such that $S_5\circ S_1$ and $S_5\circ S_1'$ lie in the stabilizer of the origin and induce rotations of $\mathbb{H}^2$ about $(0,0,1)$ coinciding at $\infty$. We deduce that $S_5\circ S_1=S_5\circ S_1'$, from where $S'_1=S_1$, and hence $S_2'=S_2$.
\end{proof}

This factorization gives the desired parameterization of the level sets of $\Lambda$ by means of $G$.  Although the $[\Sigma_u]_S$ may depend upon the representative $\Sigma_u$ in its isometry class, the family $\{[\Sigma_u]_S:S\in G\}$ does not change (it is not difficult to show that, for each $T\in\Iso(\Nil(\tau))$, there is a bijective map $\sigma:G\to G$ such that $[T(\Sigma_u)]_S=[\Sigma_u]_{\sigma(S)}$ for all $S\in G$ and $\Sigma_u\in\mathcal E(\Nil(\tau),0)$).

\begin{theorem}\label{thm:orbit}
The family of isometry classes of entire minimal graphs in $\Nil(\tau)$ with the same Abresch--Rosenberg differential as $\Sigma_u\in\mathcal{E}(\Nil(\tau),0)$ is given by
\[\Lambda^{-1}(\Lambda([\Sigma_u]))=\{[\Sigma_u]_S:S\in G\},\]
and depends on two real parameters unless the dual graph of $\Sigma_u$ is invariant under some 1-parameter group of hyperbolic or parabolic rotations or screw motions.
\end{theorem}

\begin{proof}
Let $S\in\Iso_\uparrow(\mathbb{L}^3)$ and factorize it as $S=S_1\circ S_2$ with $S_1\in\Iso_\xi(\mathbb{L}^3)$ and $S_2\in G$ by means of Lemma~\ref{lemma:orbit1}. Denote by $\Sigma_v$ the dual of $\Sigma_u$ and consider $T_1\in\Iso_\xi(\Nil(\tau))$ such that $R(T_1)=S_1$, where $R$ is defined by~\eqref{eqn:R}. Then Proposition~\ref{prop:transform} says that the dual of $S(\Sigma_v)$ is the image by $T_1$ of the dual of $S_2(\Sigma_v)$. This proves that $[\Sigma_u]_S=[\Sigma_u]_{S_2}$.

Given distinct $S,S'\in G$, we get that $[\Sigma_u]_S=[\Sigma_u]_{S'}$ if and only if $(S'\circ S^{-1})(\Sigma_v)=S_0(\Sigma_v)$ for some $S_0\in\Iso_\xi(\mathbb{L}^3)$. Hence the 2-parameter family degenerates when $\Sigma_v$ is invariant under a 1-parameter subgroup of $\Iso_\uparrow(\mathbb{L}^3)$ not contained in $\Iso_\xi(\mathbb{L}^3)$. Among the 1-parameter subgroups of $\Iso_\uparrow(\mathbb{L}^3)$, only hyperbolic and parabolic screw-motions (and, in particular, rotations) satisfy this condition.
\end{proof}

\subsection{Invariant $\tau$-surfaces in $\mathbb{L}^3$}
We will briefly discuss here the exceptional cases given by Theorem~\ref{thm:orbit}. As for computing, we will use the twin relations~\eqref{eqn:twin} between a minimal graph $\Sigma_u\subset\Nil(\tau)$ and a spacelike $\tau$-graph $\Sigma_v\subset\mathbb{L}^3$, given by
\begin{equation}\label{eqn:twin:nil}
\begin{aligned}
u_x&=\frac{v_y}{\sqrt{1-|\nabla v|^2}}-\tau y,& v_x=\frac{-u_y+\tau x}{\sqrt{1+|Gu|^2}},\\
u_y&=\frac{-v_x}{\sqrt{1-|\nabla v|^2}}+\tau x,& v_y=\frac{u_x+\tau y}{\sqrt{1+|Gu|^2}},
\end{aligned}
\end{equation}
where $Gu=\alpha\partial_x+\beta\partial_y$.
\begin{itemize}
	\item \textbf{Parabolic screw-motions.} If $\Sigma$ is an entire $\tau$-graph with $\tau\neq 0$ and invariant under parabolic screw motions, then Fern\'{a}ndez and L\'{o}pez~\cite[Lemma~3.1]{FL} proved that it is invariant under parabolic rotations (i.e., the pitch of the screw motion vanishes), so Hano and Nomizu's classification~\cite{HN} ensures that $\Sigma$ is an entire graph parameterized by 
\begin{equation}\label{eqn:parabolic-examples}
\phantom{aa}\qquad(x,w)\mapsto\left(x,\frac{1}{\sqrt{2}}\left(\frac{x^2}{2w}-w\right)+\frac{f(w)}{4\sqrt{2}\,\tau^2},\frac{-1}{\sqrt{2}}\left(\frac{x^2}{2w}+w\right)-\frac{f(w)}{4\sqrt{2}\,\tau^2}\right),
\end{equation} 
and there are essentially three families of entire solutions in terms of a parameter $a>0$, which correspond to the following choices of $f(w)$ in~\eqref{eqn:parabolic-examples}:
\begin{enumerate}
	\item $f(w)=\frac{w}{w^2-a^2}-\frac{1}{2a}\log(\frac{w-a}{w+a})$ for $w>a$,
	\item $f(w)=\frac{w}{w^2+a^2}-\frac{1}{a}\arctan(\frac{w}{a})$ for $w>0$.
	\item $f(w)=\frac{2}{w}$ for $w>0$ (this is the hyperboloid $z=(\tau^{-2}+x^2+y^2)^{1/2}$).
\end{enumerate}
Note that the parameter $b$ in~\cite{HN} is nothing but a translation in the direction of the lightlike vector $(0,1,1)$, which lies in $\Iso_\xi(\mathbb{L}^3)$.

\item \textbf{Hyperbolic screw-motions.} A complete classification seems to be still an open problem in this case. In the particular case of hyperbolic rotations, $\Sigma_v\subset\mathbb{L}^3$ is invariant under~\eqref{eqn:hyperbolic-rotation} if and only if $v(x,y)=\sqrt{f(x)^2+y^2}$ for some $f\in C^\infty(\R)$. If $\Sigma_v$ has constant mean curvature $\tau$, then this includes the totally umbilical hyperboloid ($f(x)=\tau^{-2}+x^2$), the hyperbolic cylinder ($f(x)=\frac{1}{4}\tau^{-2}$), and the semitrough given by the parametrization
\begin{align*}
 (x,y)&\mapsto \tfrac{1}{H}\left(x-\tfrac{1}{2}\coth(x),\tfrac{1}{2}\coth(x)\sinh(y),\tfrac{1}{2}\coth(x)\cosh(y)\right).
\end{align*}
The dual to $\tau$-graphs of the form $v(x,y)=\sqrt{f(x)^2+y^2}$ in $\mathbb{L}^3$ are minimal graphs in $\Nil(\tau)$ of the form $u(x,y)=y g(x)$ for some $g\in C^\infty(\R)$, as noticed by Lee~\cite[Example~4]{Lee}. Some of them have also been studied by Daniel~\cite{Dan2}, and include umbrellas and some invariant surfaces, which will be discussed in Section~\ref{sec:examples}.
\end{itemize}

\subsection{Other examples}\label{sec:examples}

In this section we will briefly discuss the explicit known examples of entire minimal graphs in Heisenberg space. 
\begin{itemize}
 \item \textbf{Umbrellas.} Let $\Sigma_u$ be the entire minimal graph defined by $u(x,y)=0$, which is the union of all geodesics passing through the origin. Its dual surface, which can be computed using~\eqref{eqn:twin:nil}, is the hyperboloid $\Sigma_v$ given by $v(x,y)=(\tau^{-2}+x^2+ y^2)^{1/2}$. Note that $\Sigma_u$ has zero Abresch--Rosenberg differential by~\eqref{eqn:Qei:R}, so $\Sigma_v$ has zero Hopf differential by Theorem~\ref{thm:differential}, which is well known since $\Sigma_v$ is totally umbilical in $\mathbb{L}^3$. Moreover, $\Sigma_u$ and $\Sigma_v$ are conformally hyperbolic, see Lemma~\ref{lema:parabolic-Q0}.

 Moreover $\Sigma_v$ is invariant under the action of $G$, since it is invariant under the 1-parameter groups of hyperbolic and parabolic rotations given by~\eqref{eqn:hyperbolic-rotation} and~\eqref{eqn:parabolic-rotation}. Theorem~\ref{thm:orbit} implies that the only entire minimal graphs with zero Abresch--Rosenberg differential are those congruent to $\Sigma_u$, i.e., the affine planes $u(x,y)=ax+by+c$ for some $a,b,c\in\R$.

 \item \textbf{Invariant surfaces.} Let us consider the entire minimal graph $\Sigma_u$ given by the function $u(x,y)=-\tau xy$ (the minus sign appears just to simplify the computations because of the particular choice of the group $G$). The dual graph is the hyperbolic cylinder $\Sigma_v$ with $v(x,y)=(\frac{1}{4}\tau^{-2}+x^2)^{1/2}$. Then $\Sigma_u$ and $\Sigma_v$ are parabolic with differential $c\,\df\zeta^2$ with respect to a conformal parameter $\zeta$, where $c\in\C$ is a non-zero constant (the value of $c$ is irrelevant since changing it is nothing but a conformal reparametrization), see~\cite{Dan2}. Let us now apply hyperbolic and parabolic rotations.

On the one hand, for each $\theta\in\R$, the hyperbolic rotation~\eqref{eqn:hyperbolic-rotation} produces dual entire graphs $\Sigma_{u_\theta}$ and $\Sigma_{v_\theta}$ with
 	\begin{equation}\label{eqn:invariant}
 	\begin{aligned}
 	\ \qquad\quad\qquad u_\theta(x,y)&=-\tau xy+\frac{\sinh(\theta)}{4\tau}\left(2\tau x\sqrt{1+4\tau^2x^2}+\mathrm{arcsinh}(2\tau x)\right),\\
 	v_\theta(x,y)&=\frac{1}{\cosh(\theta)}\sqrt{\frac{1}{4\tau^2}+x^2}+\tanh(\theta)y.
 	\end{aligned}
 	\end{equation}
 	Observe that the family $\Sigma_{u_\theta}$ is invariant under the 1-parameter group $(x,y,z)\mapsto(x,y+t,z+\tau t x)$ of translations of $\Nil(\tau)$, see also~\cite{FMP}. This corresponds to the invariance of $\Sigma_{v_\theta}$ under the 1-parameter group of translations of $\mathbb{L}^3$ given by $(x,y,z)\mapsto(x,y+t,z-t\tanh(\theta))$, see Corollary~\ref{coro:transform}.

On the other hand, for each $a\in\R$, we can also apply the parabolic rotation~\eqref{eqn:parabolic-rotation} to $\Sigma_v$, in which case we reach the transformed graph
 	\[\qquad\quad\overline v_a(r,s)=\frac{2+4a^2+a^4}{\sqrt{a^4+4}}\sqrt{\frac{1}{4\tau^2}+r^2}-\frac{2a(a^2+2)}{\sqrt{a^4+4}}r+\frac{a^2}{\sqrt{a^4+4}}s\]
 	where we have considered the rotation of parameters
 	\[r=\frac{(a^2-2)x-2ay}{\sqrt{a^4+4}},\qquad s=\frac{2ax+(a^2-2)y}{\sqrt{a^4+4}}.\]
 	Since $\Sigma_{\overline v_a}$ is invariant under the 1-parameter group of translations 
 	\[(r,s,z)\mapsto\left(r,s+t,z-\frac{a^2}{\sqrt{a^2+4}}t\right),\]
 	Corollary~\ref{coro:transform} ensures that the dual surface must be invariant under a 1-parameter group of translations in $\Nil(\tau)$, so it must be congruent to some $\Sigma_{u_{\theta}}$ by the classification of invariant surfaces, see~\cite{FMP}.
\end{itemize} 
\vspace{5pt}
These examples can be summarized in the statement below, which was already proved in~\cite{AR,Dan2,FM} in different ways. We remark that our technique can be applied to other cases that will not be explicitly treated here (e.g., to the semitrough, which is hyperbolic and has Hopf differential $c\,\df\zeta^2$).

\begin{corollary}\label{coro:constant-Q}
Let $\Sigma\subset\Nil(\tau)$ be a parabolic minimal graph with Abresch--Rosenberg differential $Q$.
\begin{enumerate}[label=(\alph*)]
	\item $Q=0$ if and only if $\Sigma$ is an umbrella;
	\item If $\Sigma$ is parabolic, then $Q=c\,\df\zeta^2$ for some $c\in\C$ if and only if $\Sigma$ is invariant under a 1-parameter group of isometries of $\Nil(\tau)$. 
\end{enumerate}
\end{corollary}

\begin{remark}
The assumption that the surface is graphical is important in Corollary~\eqref{coro:constant-Q}. For instance, consider the helicoid parameterized by 
\begin{equation}\label{eqn:helicoid}
(x,y)\in\R^2\mapsto (f(x)\cos(y),f(x)\sin(y),a y),
\end{equation}
where $f:\R\to\R$ satisfies the \textsc{ode} $f'(x)=\sqrt{(\tau f(x)^2-a)^2+f(x)^2}$, for some $a>0$. The parametrization~\eqref{eqn:helicoid} is conformal, so the helicoid is parabolic (which contrasts with the hyperbolic limit case $a=0$), and its Abresch--Rosenberg differential is given by $Q=c\,\df\zeta^2$ with respect to a global conformal parameter $\zeta$, see~\cite{Dan2}. 
\end{remark}

\section{Curvature estimates for entire minimal graphs in $\Nil(\tau)$}
In this last section, we will use the duality to give a curvature estimate for entire minimal graphs in $\Nil(\tau)$, with some applications. We will begin by proving a nice formula relating the determinants of the Hessian matrices of dual graphs.

\begin{lemma}\label{lema:hessian}
Let $\Sigma_u\subset\Nil(\tau)$ and $\Sigma_v\subset\mathbb{L}^3$ be dual graphs, and let us denote by $\widetilde K$ the Gauss curvature of $\Sigma_v$. Then
\begin{equation}\label{eqn:hessian}
u_{xy}^2-u_{xx}u_{yy}=\frac{v_{xy}^2-v_{xx}v_{yy}}{(1-v_x^2-v_y^2)^2}+\tau^2=\widetilde{K}+\tau^2.
\end{equation}
Hence the following dual statements hold:
\begin{itemize}
 \item[(a)] If $\Sigma_u\subset\Nil(\tau)$ is an entire minimal graph, then
 \[0\leq u_{xy}^2-u_{xx}u_{yy}\leq \tau^2.\]
 \item[(b)] If $\Sigma_v\subset\mathbb{L}^3$ is an entire spacelike $\tau$-graph, then
\[-\tau^2\leq\widetilde{K}\leq 0.\]
\end{itemize}
\end{lemma}

\begin{proof}
Let $Q=Q_{1,-2i\tau}$ be the Abresch--Rosenberg differential of $\Sigma_u$, and let $\widetilde Q=\widetilde{Q}_{-i,0}$ be the Hopf differential of $\Sigma_v$, as in Remark~\ref{rmk:differential:nil}. Taking squared moduli in the first equations of~\eqref{eqn:Qei:R} and~\eqref{eqn:Qei:L}, and reducing the resulting expressions by means of~\eqref{eqn:mean-curvature}, we reach the identities
\begin{equation}\label{lema:hessian:eqn1}
\begin{aligned}
 |Q(e_1,e_1)|^2&=4(1+\alpha^2)^2\,\frac{u_{xy}^2-u_{xx}u_{yy}}{\omega^4},\\
 |\widetilde Q(\widetilde e_1,\widetilde e_1)|^2&=4(1-\widetilde\alpha^2)^2\left(\frac{v_{xy}^2-v_{xx}v_{yy}}{\widetilde\omega^4}+\tau^2\right).
\end{aligned}
\end{equation}
Since $Q(e_1,e_1)=\widetilde Q(\widetilde e_1,\widetilde e_1)$ by Theorem~\ref{thm:differential}, and $\frac{1+\alpha^2}{\omega^2}=\frac{\omega^2-\beta^2}{\omega^2}=1-\widetilde\alpha^2$ by the twin relations, we deduce that Equation~\eqref{eqn:hessian} holds true.

If $\Sigma_v$ is an entire spacelike $\tau$-graph in $\mathbb{L}^3$, then its second fundamental form $\widetilde \sigma$ satisfies $2\tau^2\leq|\widetilde \sigma|^2\leq 4\tau^2 $, see~\cite{CY2,CT,Treibergs}. Item (b) follows from the well-known identity $|\widetilde\sigma|^2=4\tau^2+2\widetilde K$, and item (a) is dual to (b) taking into account~\eqref{eqn:hessian}.
\end{proof}

\begin{remark}\label{rmk:q}
Espinar and Rosenberg~\cite{ER2} defined the modulus of the Abresch--Rosenberg differential $Q\,\df\zeta^2$ of an $H$-surface $\Sigma$ in $\mathbb{E}(\kappa,\tau)$, in terms of a complex conformal parameter $\zeta=r+is$, as the function $q=\frac{4|Q|^2}{\mu^4}$, where $\mu$ is the conformal factor such that the metric of $\Sigma$ reads $\mu^2(\df r^2+\df s^2)$. Hence $q$ does not depend upon the choice of the parameter $\zeta$. In the case of a minimal graph $\Sigma_u\subset\Nil(\tau)$, combining~\cite[Lemma~2.2]{ER2}, and Equations~\eqref{eqn:hessian} and~\eqref{eqn:negative-K1}, it is not difficult to get
\begin{equation*}{}
\frac{q}{4\tau^2\nu^4}=u_{xy}^2-u_{xx}u_{yy},
\end{equation*}
which shows that the Hessian determinant is a geometric quantity for an entire minimal graph in $\Nil(\tau)$.
On the other hand, the right-hand-side of~\eqref{eqn:hessian} is nothing but the modulus of the Hopf differential of $\Sigma_v$, so that Equation~\eqref{eqn:hessian} can be understood as the equality between the moduli of the holomorphic differentials.
\end{remark}

\begin{lemma}\label{lemma:K}
Let $\Sigma_u\subset\mathbb{E}(\kappa,\tau)$ be an entire $H$-graph with critical mean curvature. If $K$ and $\nu$ denote the Gauss curvature and the angle function of $\Sigma_u$, then
\begin{equation}\label{eqn:Knu}
-4(H^2+\tau^2)\nu^2\leq K\leq -3(H^2+\tau^2)\nu^4<0.
\end{equation}
\end{lemma}

\begin{proof}
Since Daniel correspondence~\cite{Dan} preserves $K$ and $\nu$, as well as the value of $H^2+\tau^2$, we will assume that $\Sigma_u$ is an entire minimal graph in $\Nil(\tau)$ with no loss of generality, and prove~\eqref{eqn:Knu} with $H=0$. 

Let $\Sigma_v\subset\mathbb{L}^3$ be the dual entire $\tau$-graph, and let $\Phi:\Sigma_u\to\Sigma_v$ be the conformal diffeomorphism given by~\eqref{eqn:phi}. Since $\Phi$ induces a conformal factor $\nu^2$ between $\Sigma_u$ and $\Sigma_v$, see Equation~\eqref{eqn:conformal-factor}, the Gauss curvature $\widetilde{K}$ of $\Sigma_v$ satisfies $\widetilde{K}\nu^2=K-\Delta\log(\nu)$, where the Laplacian is computed on $\Sigma_u$ (this identity and those hereafter only make sense through the diffeomorphism $\Phi$).

The identity $\Delta\log(\nu)=\frac{1}{\nu}\Delta\nu-\frac{1}{\nu^2}\|\nabla\nu\|^2$, along with the fact that $\nu$ lies in the kernel of the stability operator of $\Sigma_u$ (i.e., $\Delta\nu=2K\nu+4\tau^2\nu^3$), let us deduce that
\begin{equation}\label{eqn:negative-K1}
\widetilde{K}=\frac{-K}{\nu^2}-4\tau^2+\frac{\|\nabla\nu\|^2}{\nu^4}.
\end{equation}
The first inequality in the statement is a consequence of estimating $\widetilde K\leq 0$ and $\|\nabla\nu\|\geq 0$ in~\eqref{eqn:negative-K1} in virtue of Lemma~\ref{lema:hessian}.

As for the second inequality, we will need a bound of $\|\nabla\nu\|^2$ from above. In the sequel, $\|\cdot\|_1$ and $\langle\cdot,\cdot\rangle_1$ will stand for the norm and inner product in $\mathbb{L}^3$, respectively. Also, $\nabla$ and $\widetilde\nabla$ will denote the gradient operators in $\Sigma_u$ and $\Sigma_v$, respectively. Using again that the conformal factor between $\Sigma_u$ and $\Sigma_v$ is $\nu^2$, we obtain 
\begin{equation}\label{eqn:negative-K2}
\frac{1}{\nu^4}\|\nabla\nu\|^2=\frac{1}{\nu^2}\|\widetilde\nabla\nu\|_1^2=\nu^2\|\widetilde\nabla\tfrac{1}{\nu}\|_1^2,
\end{equation}
The angle function of $\Sigma_v$ is given by $\omega=\langle\widetilde{N},\partial_z\rangle_1$, where $\widetilde N$ is the upward-pointing unit normal to $\Sigma_v$, and satisfies $\omega=\frac{1}{\nu}$. Moreover, $\widetilde\nabla\omega=\widetilde{A}\widetilde{T}$, where $\widetilde A$ is the shape operator of $\Sigma_v$, and $\widetilde T=\partial_z+\omega\widetilde N$ is the tangent part to $\Sigma_v$ of the (timelike) unit Killing vector field $\partial_z$ in $\mathbb{L}^3$. Plugging this information in~\eqref{eqn:negative-K2}, we obtain
\begin{equation}\label{eqn:negative-K3}
\frac{1}{\nu^4}\|\nabla\nu\|^2=\nu^2\|\widetilde A\widetilde T\|_1^2\leq\nu^2\|\widetilde A\|^2\|\widetilde T\|_1^2=(4\tau^2+2\widetilde K)(1-\nu^2)
\end{equation}
where we have used the  identity $\|\widetilde A\|^2=4\tau^2+2\widetilde K$, as well as the fact that $\nu^2\|\widetilde T\|^2_1=\nu^2(-1+\omega^2)=1-\nu^2$. Combining~\eqref{eqn:negative-K3} and~\eqref{eqn:negative-K1}, we reach the inequality
\begin{equation}\label{eqn:negative-K4}
K\leq -4\tau^2\nu^4+\widetilde K\nu^2(1-2\nu^2).
\end{equation}
Note that~\eqref{eqn:negative-K4} holds for any (not necessarily entire) minimal graph in $\Nil(\tau)$, and our result will be a consequence of Lemma~\ref{lema:hessian}. Given $p\in\Sigma$, let us distinguish two cases. If $\nu(p)\geq\frac{1}{\sqrt{2}}$, then $1-2\nu(p)^2\leq 0$, so we have $\widetilde K\nu^2(1-2\nu^2)\leq -\tau^2\nu^2(1-2\nu^2)$ at $p$. If, on the contrary, $\nu(p)\leq\frac{1}{\sqrt{2}}$, then $1-2\nu(p)^2\geq 0$, and we get $\widetilde K\nu^2(1-2\nu^2)\leq 0$ at $p$. Hence~\eqref{eqn:negative-K4} yields the estimate
\[K\leq\begin{cases}
-2\tau^2\nu^4-\tau^2\nu^2& \text{if }\nu\geq\frac{1}{\sqrt{2}},\\
-4\tau^2\nu^4& \text{if }\nu\leq\frac{1}{\sqrt{2}}.
\end{cases}\]
In both cases it follows that $K\leq -3\tau^2\nu^4$, so we are done.
\end{proof}

Observe that vertical planes are the only complete flat minimal surfaces in $\Nil(\tau)$, so they have finite total curvature, see~\cite[Theorem~3.1]{ER2}. It is conjectured in~\cite{MPR} that they are the only complete orientable stable minimal surfaces with intrinsic quadratic area growth (i.e., such that the area of intrinsic metric balls grow at most quadratically with respect to their radius). The above estimate allows us to rewrite this conjecture from the point of view of curvature.

\begin{theorem}\label{thm:curvature}
Any entire graph $\Sigma\subset\mathbb{E}(\kappa,\tau)$ with critical mean curvature has negative Gauss curvature. In particular, it has at least quadratic intrinsic area growth, and the following assertions are equivalent:
\begin{enumerate}[label=(\roman*)]
 \item $\Sigma$ has quadratic intrinsic area growth.
 \item $\Sigma$ has finite total curvature.
\end{enumerate}
\end{theorem}

\begin{proof}
The fact that $\Sigma$ has at least quadratic area growth follows from~\eqref{eqn:Knu} by comparing the area growth of $\Sigma$ with the (quadratic) area growth of the Euclidean plane. Moreover, Li~\cite{Li} showed that quadratic area growth plus negative curvature implies finite total curvature. The converse holds even if the curvature changes sign, as proved by Hartman~\cite{Hartman}.
 \end{proof}

\begin{remark}
Equation~\eqref{eqn:Knu} does not give much information about the integrability of $K$, since $\nu^2$ is never an integrable function on an entire minimal graph~\cite{ManNel}, and $\nu^4$ can be either integrable or not, as we shall see in the next examples.
\begin{itemize}
 \item The umbrella $\Sigma_u$ given by $u(x,y)=0$ is rotationally invariant, and its Gauss curvature and angle function are given by
 \[K(r)=-\frac{3\tau^2+2\tau^4r^2}{(1+\tau^2r^2)^2},\qquad \nu(r)=\frac{1}{\sqrt{1+\tau^2r^2}},\]
 where $r=(x^2+y^2)^{1/2}$. Hence $K\not\in L^1(\Sigma_u)$ but $\nu^4\in L^1(\Sigma_u)$. There is a functional dependence between $K$ and $\nu$ given by $K=-\tau^2\nu^4-2\tau^2\nu^2$.
 \item The invariant surfaces $\Sigma_{u_\theta}$ given by~\eqref{eqn:invariant} satisfy 
 \[K(x)=\frac{-4\tau^2}{\cosh^2(\theta)(1+4\tau^2x^2)^2},\qquad\nu(x)=\frac{1}{\cosh(\theta)\sqrt{1+4\tau^2x^2}}.\]
 In this case, $K\not\in L^1(\Sigma_{u_\theta})$ and $\nu\not\in L^1(\Sigma_{u_\theta})$. The relation between $K$ and $\nu$ for these surfaces reads $K=-4\tau^2\cosh^2(\theta)\nu^4$.
\end{itemize}
\end{remark}

\begin{remark}\label{rmk:catenoids-helicoids}
It is not true in general that minimal surfaces in $\Nil(\tau)$ have negative Gauss curvature. Note that Gauss equation is given by $K=\det(A)+\tau^2-4\tau^2\nu^2$ in this context, as a particular case of~\eqref{eqn:gauss}. Let us discuss a couple of examples showing that non-graphical surfaces may display different behaviors:
\begin{itemize}
	\item\textbf{Catenoids.} The $1$-parameter family of catenoids in $\Nil(\tau)$ is given in terms of a parameter $E>0$ by the function $u(x,y)=h(r)$, where $r=(x^2+y^2)^{1/2}$ and $h:[E,\infty)\to\R$ is the solution to the \textsc{ode}
 \[h'(r)=\frac{E\sqrt{1+\tau^2r^2}}{\sqrt{r^2-E^2}},\qquad h(E)=0.\]
 The Gauss curvature and the angle function can be computed as
 \[K(r)=-\frac{E^2+3\tau^2r^4+2\tau^4r^6}{r^4(1+\tau^2r^2)^2}<0,\quad \nu(r)=\frac{\sqrt{r^2-E^2}}{r\sqrt{1+\tau^2r^2}}\]
 It follows that $K$ is not integrable for any $E>0$.
 
 \item\textbf{Helicoids.} The 1-parameter family of complete minimal helicoids parameterized by $(x,y)\in\R^2\mapsto(x\cos(y),x\sin(y),ay)$ (for some $a>0$) has constant Gauss curvature $\frac{2a\tau-1}{a^2}$ along the $z$-axis, so its sign depends on $a$.
\end{itemize}
\end{remark}

As a conclusion, we will prove a simple necessary condition in order to have finite total curvature. Umbrellas are examples showing that it is not sufficient.

\begin{proposition}\label{prop:tendToZero}
If $\Sigma\subset\mathbb{E}(\kappa,\tau)$ is a complete stable $H$-surface with finite total curvature, then both the Gauss curvature and the angle function tend to zero on any divergent sequence of points of $\Sigma$.
\end{proposition}

\begin{proof}
Let $\{p_n\}\subset\Sigma$ be a divergent sequence in $\Sigma$. Given $n\in\mathbb{N}$, consider $\Sigma_n$ to be a translation of $\Sigma$ taking $p_n$ to the origin. Since the geometry of $\mathbb{E}(\kappa,\tau)$ is bounded and each $\Sigma_n$ is stable, standard convergence arguments (see~\cite{RST}) implies that $\{\Sigma_n\}$ subconverges to a complete $H$-surface $\Sigma_\infty$, which is either a complete multigraph or a vertical cylinder over a curve of geodesic curvature $\pm 2H$.

If $K(p_n)$ does not converge to zero, the subsequence may be chosen so that $\Sigma_\infty$ has Gauss curvature different from zero at the origin, so there is a relatively compact neighborhood $U_\infty\subset\Sigma_\infty$ such that $\int_{U_\infty}K_\infty\neq 0$. Since the convergence is in the $C^k$-topology on compact subsets for any $k\geq 0$, we deduce that there exists a sequence of relatively compact open subsets $U_n\subset\Sigma$ such that $p_n\in U_n$ and $\int_{U_n}K$ becomes arbitrarily close to $\int_{U_\infty}K_\infty\neq 0$. This clearly implies that $\Sigma$ does not have finite total curvature, contradicting the statement.

Now assume that there is a sequence $\{p_n\}\subset\Sigma$ such that $\nu(p_n)$ converges to some value $a$. Applying the same reasoning as above, we conclude that there a subsequence converging to a complete surface $\Sigma_\infty$, but item (a) implies that $\Sigma_\infty$ is a complete flat $H$-surface, so $\Sigma_\infty$ is a vertical cylinder by~\cite[Theorem 3.1]{ER2}. Hence $a=0$ since the angle function of a vertical cylinder vanishes identically.
\end{proof}

\end{document}